\documentclass[12pt]{amsart}
\usepackage{amsmath,amssymb,amsbsy,amsfonts,latexsym,amsopn,amstext,cite,
                                               amsxtra,euscript,amscd,bm}
\usepackage{url}
\usepackage{mathrsfs}

\usepackage{color}
\usepackage[colorlinks,linkcolor=blue,anchorcolor=blue,citecolor=blue,backref=page]{hyperref}
\usepackage{color}
\usepackage{graphics,epsfig}
\usepackage{graphicx}
\usepackage{float}
\usepackage{epstopdf}
\hypersetup{breaklinks=true}

\usepackage[top=1in, bottom=1in, left=1in, right=1in]{geometry}

\usepackage[np]{numprint}
\npdecimalsign{\ensuremath{.}}

\def\Th1{\varTheta}

\usepackage{bibentry}
\usepackage{bbm}

\usepackage[english]{babel}
\usepackage{mathtools}
\usepackage{todonotes}
\usepackage{url}
\usepackage[colorlinks,linkcolor=blue,anchorcolor=blue,citecolor=blue,backref=page]{hyperref}

\usepackage[norefs,nocites]{refcheck}

\begin{document}

\newtheorem{theorem}{Theorem}
\newtheorem{lemma}[theorem]{Lemma}
\newtheorem{claim}[theorem]{Claim}
\newtheorem{cor}[theorem]{Corollary}
\newtheorem{conj}[theorem]{Conjecture}
\newtheorem{prop}[theorem]{Proposition}
\newtheorem{definition}[theorem]{Definition}
\newtheorem{question}[theorem]{Question}
\newtheorem{example}[theorem]{Example}
\newcommand{\hh}{{{\mathrm h}}}
\newtheorem{remark}[theorem]{Remark}

\numberwithin{equation}{section}
\numberwithin{theorem}{section}
\numberwithin{table}{section}
\numberwithin{figure}{section}

\def\sssum{\mathop{\sum\!\sum\!\sum}}
\def\ssum{\mathop{\sum\ldots \sum}}
\def\iint{\mathop{\int\ldots \int}}

\newcommand{\diam}{\operatorname{diam}}

\def\squareforqed{\hbox{\rlap{$\sqcap$}$\sqcup$}}
\def\qed{\ifmmode\squareforqed\else{\unskip\nobreak\hfil
\penalty50\hskip1em \nobreak\hfil\squareforqed
\parfillskip=0pt\finalhyphendemerits=0\endgraf}\fi}

\newfont{\teneufm}{eufm10}
\newfont{\seveneufm}{eufm7}
\newfont{\fiveeufm}{eufm5}
%
%
\newfam\eufmfam
     \textfont\eufmfam=\teneufm
\scriptfont\eufmfam=\seveneufm
     \scriptscriptfont\eufmfam=\fiveeufm
%
%
\def\frak#1{{\fam\eufmfam\relax#1}}

\newcommand{\bflambda}{{\boldsymbol{\lambda}}}
\newcommand{\bfmu}{{\boldsymbol{\mu}}}
\newcommand{\bfxi}{{\boldsymbol{\eta}}}
\newcommand{\bfrho}{{\boldsymbol{\rho}}}

\def\eps{\varepsilon}

\def\fK{\mathfrak K}
\def\fT{\mathfrak{T}}
\def\fL{\mathfrak L}
\def\fR{\mathfrak R}

\def\fA{{\mathfrak A}}
\def\fB{{\mathfrak B}}
\def\fC{{\mathfrak C}}
\def\fM{{\mathfrak M}}
\def\fS{{\mathfrak  S}}
\def\fU{{\mathfrak U}}
\def\fW{{\mathfrak W}}

\def\T {\mathsf {T}}
\def\Tor{\mathsf{T}_d}
\def\Tore{\widetilde{\mathrm{T}}_{d} }

\def\sM {\mathsf {M}}

\def\ss{\mathsf {s}}

\def\Kmnd{\cK_d(m,n)}
\def\Kmnp{\cK_p(m,n)}
\def\Kmnq{\cK_q(m,n)}

\def \balpha{\bm{\alpha}}
\def \bbeta{\bm{\beta}}
\def \bgamma{\bm{\gamma}}
\def \bdelta{\bm{\delta}}
\def \bzeta{\bm{\zeta}}
\def \blambda{\bm{\lambda}}
\def \bchi{\bm{\chi}}
\def \bphi{\bm{\varphi}}
\def \bpsi{\bm{\psi}}
\def \bnu{\bm{\nu}}
\def \bomega{\bm{\omega}}

\def \bell{\bm{\ell}}

\def\eqref#1{(\ref{#1})}

\def\vec#1{\mathbf{#1}}

\newcommand{\abs}[1]{\left| #1 \right|}

\def\Zq{\mathbb{Z}_q}
\def\Zqx{\mathbb{Z}_q^*}
\def\Zd{\mathbb{Z}_d}
\def\Zdx{\mathbb{Z}_d^*}
\def\Zf{\mathbb{Z}_f}
\def\Zfx{\mathbb{Z}_f^*}
\def\Zp{\mathbb{Z}_p}
\def\Zpx{\mathbb{Z}_p^*}
\def\cM{\mathcal M}
\def\cE{\mathcal E}
\def\cH{\mathcal H}

\def\le{\leqslant}

\def\ge{\geqslant}

\def\sfB{\mathsf {B}}
\def\sfC{\mathsf {C}}
\def\L{\mathsf {L}}
\def\FF{\mathsf {F}}

\def\sE {\mathscr{E}}
\def\sS {\mathscr{S}}

\def\cA{{\mathcal A}}
\def\cB{{\mathcal B}}
\def\cC{{\mathcal C}}
\def\cD{{\mathcal D}}
\def\cE{{\mathcal E}}
\def\cF{{\mathcal F}}
\def\cG{{\mathcal G}}
\def\cH{{\mathcal H}}
\def\cI{{\mathcal I}}
\def\cJ{{\mathcal J}}
\def\cK{{\mathcal K}}
\def\cL{{\mathcal L}}
\def\cM{{\mathcal M}}
\def\cN{{\mathcal N}}
\def\cO{{\mathcal O}}
\def\cP{{\mathcal P}}
\def\cQ{{\mathcal Q}}
\def\cR{{\mathcal R}}
\def\cS{{\mathcal S}}
\def\cT{{\mathcal T}}
\def\cU{{\mathcal U}}
\def\cV{{\mathcal V}}
\def\cW{{\mathcal W}}
\def\cX{{\mathcal X}}
\def\cY{{\mathcal Y}}
\def\cZ{{\mathcal Z}}
\newcommand{\rmod}[1]{\: \mbox{mod} \: #1}

\def\cg{{\mathcal g}}

\def\vy{\mathbf y}
\def\vr{\mathbf r}
\def\vx{\mathbf x}
\def\va{\mathbf a}
\def\vb{\mathbf b}
\def\vc{\mathbf c}
\def\ve{\mathbf e}
\def\vh{\mathbf h}
\def\vk{\mathbf k}
\def\vm{\mathbf m}
\def\vz{\mathbf z}
\def\vu{\mathbf u}
\def\vv{\mathbf v}

\def\e{{\mathbf{\,e}}}
\def\ep{{\mathbf{\,e}}_p}
\def\eq{{\mathbf{\,e}}_q}

\def\Tr{{\mathrm{Tr}}}
\def\Nm{{\mathrm{Nm}}}

 \def\SS{{\mathbf{S}}}

\def\lcm{{\mathrm{lcm}}}

 \def\0{{\mathbf{0}}}

\def\({\left(}
\def\){\right)}
\def\l|{\left|}
\def\r|{\right|}
\def\fl#1{\left\lfloor#1\right\rfloor}
\def\rf#1{\left\lceil#1\right\rceil}
\def\fl#1{\left\lfloor#1\right\rfloor}
\def\ni#1{\left\lfloor#1\right\rceil}
\def\sumstar#1{\mathop{\sum\vphantom|^{\!\!*}\,}_{#1}}

\def\mand{\qquad \mbox{and} \qquad}

\def\tblue#1{\begin{color}{blue}{{#1}}\end{color}}




\hyphenation{re-pub-lished}

\mathsurround=1pt

\def\bfdefault{b}

\def \F{{\mathbb F}}
\def \K{{\mathbb K}}
\def \N{{\mathbb N}}
\def \Z{{\mathbb Z}}
\def \P{{\mathbb P}}
\def \Q{{\mathbb Q}}
\def \R{{\mathbb R}}
\def \C{{\mathbb C}}
\def\Fp{\F_p}
\def \fp{\Fp^*}

 \def \xbar{\overline x}

\title[Roth's Theorem in Super Smooth Numbers]
{Roth's Theorem in Super Smooth Numbers}

\author [L. P. Wijaya]{Laurence P. Wijaya}
\address{Department of Mathematics, University of Kentucky, 715 Patterson Office Tower, Lexington, KY 40506, USA}
\email{laurence.wijaya@uky.edu}

\begin{abstract}  We say that the set of $y$-smooth numbers $\mathcal{S}(N,y)$ up to $N$ is super smooth if $y=\log^KN$ for a large fixed constant $K$. We show that the Roth's theorem on arithmetic progressions is true in super smooth numbers case. This extends the result of Harper where he showed the statement is true under a weaker hypothesis.
 \end{abstract}


\maketitle


\section{Introduction}
The study of arithmetic progression inside some subset of the natural numbers was started by van der Waerden in 1927 where he proved that if we color the set of natural numbers by finitely many colors, where each number can only get exactly one color, then for every positive integer $k$, there is a monochromatic arithmetic progression of length $k$. Later on, Erd\H{o}s and Tur\'{a}n \cite{ET} conjectured that any subset of natural numbers with positive upper density contains arithmetic progression of length $k$ for any $k\geq 3$. The case $k=3$ was first solved by Roth in 1953 \cite{Roth} and for the general case $k\geq 4$ was solved by Szemer\'edi in 1975 \cite{Sze}. Now there are several improvements on the quantitative results on these theorems, see for example \cite{KM, LSS}.

In fact, there is more general conjecture from Erd\H os and Tur\'an that states that if $A$ is a subset of the natural numbers $\N$ such that
\[
\sum_{a\in A}\frac1a=\infty,
\]
then there is an arithmetic progression of length $k$ inside of $A$ for any $k\geq 3$. While this conjecture still remains open for $k\geq 4$, there is known special case, such as when $A=\mathbb{P}$, where $\mathbb{P}$ is the set of prime numbers, which is a celebrated result in \cite{GT2}. The Roth's theorem in primes is proved by Green \cite{G} and now follows from Kelley and Meka's result \cite{KM}, that also implies the Erd\H{o}s-Tur\'{a}n conjecture for $k=3$. The main tools in Green's paper are the transference principle and the $W$-trick. These tools allow us to transfer the result from the dense case of the natural numbers to the relatively dense case in the primes.

Another interesting set of numbers is the so-called \textit{smooth numbers}. This subset of the natural numbers is also sparse, and in some cases it is very sparse. A positive integer $n$ is said to be \textit{$y$-smooth }for some positive integer $y$ if every prime factor of $n$ is at most $y$. There are many applications of smooth numbers, see \cite{GrSur} for an excellent survey on this topic. One would expect that for a suitable choice of $y$, the Roth's theorem holds for $y$-smooth numbers. We expect that if $N$ is parameter going to infinity and if $\mathcal{S}(N,y)$ is the number of $y$-smooth numbers up to $N$, then for suitable $y$ depending on $N$ there exists arithmetic progression of length $3$ in $\mathcal{S}(N,y)$. The case where $y=N^{\varepsilon}$ for any $\varepsilon>0$ follows from the original Roth's theorem since in this case we have that the set of $y$-smooth numbers is dense enough in $[1,N]$. One interesting case is to consider when $y$ is $\log^KN$ for some large $K$. It is known $\Psi(N,y):=|\mathcal{S}(N,y)|=N^{1-1/K+o(1)}=N^{\alpha}$ for $y=\log^KN$ as $N\rightarrow \infty$, with $\alpha=1-1/K+o(1)$. We will use this fact repeatedly in this article. For the proof, see again \cite{GrSur}. In the case when $K$ is fixed large number, we call the set $\mathcal{S}(N,y)$ to be the set of \textit{super smooth numbers}.

However, the case for $y$-smooth numbers with $y=\log^KN$ is not covered by the current result of Kelley and Meka. In 2016, Harper \cite{H} established a result where he proved that for any $\delta>0$, there is large constant $K$ depends on $\delta$ such that if $x$ is a parameter going to infinity, depending only on $\delta$, $y$ is such that $\log^Kx\leq y\leq x$, and $A$ is subset of $[1,x]\cap \N$ such that any prime factor of an element of $A$ is at most $y$ with $|A|\geq \delta\Psi(x,y)$, then $A$ contains arithmetic progression of length $3$. Note that in this result the parameter $K$ goes to infinity as $\delta$ goes to zero. He mentioned that one could get the result if $K$ is fixed large constant by using $W$-trick but never explicitly proved the statement. In this article, our main aim is to prove the previous statement.
\begin{theorem}
\label{thm:Roth Super Smooth}
    Let $K$ be a fixed large positive integer. For any $\delta>0$, for any large natural number $N$ in terms of $\delta$ and $y=\log^KN$, and any set $A\subseteq \mathcal{S}(N,y)$ such that $|A|\geq \delta\Psi(N,y)$ where $\mathcal{S}(N,y)$ is the set of $y$-smooth numbers up to $N$ and $\Psi(N,y)=|\mathcal{S}(N,y)|$, there exists arithmetic progression of length $3$ inside $A$.
\end{theorem}

We sketch the idea of the proof. We want to show that $\sum_{n,d\leq N, d\neq 0}\mathbbm{1}_A(n)\mathbbm{1}_A(n+d)\mathbbm{1}_A(n+2d)\gg 1$ where $\mathbbm{1}_A$ is the characteristic function of the set $A$, so that there exists $n,d$ such that $n,n+d,n+2d\in A$. To do this, we pass to another subset of $[N]$ that is given by the $W$-trick, and show that the arithmetic progression there corresponds to another arithmetic progression in $A$.

The outline for the rest of paper is as follows. We collect some notation in Section~\ref{sec:not} and some preliminary results in Section~\ref{sec:prel}. We prove the Theorem \ref{thm:Roth Super Smooth} in Section \ref{sec:main}, while the main tools are proven in Section \ref{sec:trick}.

\section{Notations and Conventions}
\label{sec:not}
We use the notations $f\ll g$ or $g\gg f$ or $f=O(g)$ as equivalent to the statement $|f|\leq Cg$ for some constant $C>0$. If the constant depends on a parameter, for example if it depends on $\delta,$ we will use the subscript $f=O_{\delta}(g)$ to emphasize that $C=C(\delta)$ depends only on $\delta.$ We write $f=o(g)$ if $(f/g)(x)$ tends to zero as $x$ goes to infinity. We again write the subscript if the decay depends on some parameters. We write $P(q)$ as the largest prime factor of natural number $q$, and $\omega(q)$ as number of distinct prime factors of $q$. For a positive real number $x,$ we denote $\{1,\dots,\lfloor x\rfloor\}$ by $[ x]$. Write $\log_k(N)$ as the $k$-th iterative logarithm $\log(\log(\cdots(\log N)))$.

We write $\mathcal{S}(N,y)$ to denote the set of positive $y$-smooth numbers up to $N$, and we denote its cardinality by $\Psi(N,y)$ as mentioned previously. The set of all $y$-smooth numbers is denoted by $\mathcal{S}(y)$. We also write $e(x)$ for $e^{2\pi ix}$. The notation $\Psi(N,y;q,a)$ stands for the number of elements in $\mathcal{S}(N,y)$ such that they are congruent to $a\mod q$. Another notation we use is $\Psi_q(N,y)$ that denotes the number of elements of $\mathcal{S}(N,y)$ which are coprime with $q$. One can see that
\[
\Psi_q(N,y)=\sum_{(a,q)=1}\Psi(N,y;q,a).
\]

In this article, if we say arithmetic progression, we always mean that it is a nontrivial arithmetic progression, i.e., the difference is non zero. We also write $\mathbbm{1}_A$ for the characteristic function of some specified set $A$.

For $f : \Z/N\Z\rightarrow \R$, we denote
\[
\|f\|_1:=\sum_{n\in\Z/N\Z}|f(n)|
\]
and
\[
\|f\|_{\infty}=\sup_{n\in\Z/N\Z}|f(n)|.
\]

\section{Preliminaries}
\label{sec:prel}
\subsection{Smooth Numbers Results}
We recall several results about the cardinality of smooth numbers that we need. We begin by discussing the distribution of smooth numbers in arithmetic progressions. One can see for example \cite{BaP, FT} for the early results on this topic. We record the following lemma which will be useful for our purpose, which is \cite[Theorem 1]{Gr}. Let $\Psi(x,y;q,a)$ denote the number of $y$-smooth numbers up to $x$ that are congruent to $a\mod q$ and $\Psi_q(x,y)$ denotes the number of $y$-smooth numbers $n$ up to $x$ such that $\gcd(n,q)=1$.
\begin{lemma}
\label{lem:Granville}
    Let $\beta$ be a fixed positive quantity. Then we have
    \begin{align*}
        \Psi(x,y;q,a)=\frac{1}{\varphi(q)}\Psi_q(x,y)\left( 1+O\left( \frac{\log q}{\log y} \right) \right)
    \end{align*}
    holds uniformly in the range
    \[
    2\leq y\leq x,\quad 1\leq q\leq \min\{x,y^{\beta}\},\quad \gcd(a,q)=1.
    \]
\end{lemma}
A stronger result can be found in \cite{BaP} with smaller range.

There are other extensive studies on the relation between $\Psi(x,y;q,a)$ and $\Psi_q(x,y)$. In fact, Soundararajan \cite{Sound} conjectured that for any positive real number $A$, if $y$ and $q$ are large numbers with $q\leq y^A$, then we have as $\log x/\log q\rightarrow \infty$,
\[
\Psi(x,y;q,a)\sim \frac{1}{\varphi(q)} \Psi_q(x,y)
\]
for any $(a,q)=1$.
See \cite{BS, Harp} for some progress on this conjecture. Now there are several results on the asymptotic value of $\Psi_q(x,y)$ given in \cite{BT1} and \cite{BT2}. We see that by combining \cite[Theorem 2.1]{BT1} and \cite[Corollary 2.2]{BT1} we have the following statement.
\begin{lemma}
    \label{lem : Breteche Tenenbaum result}
    Let $g_q(s):=\displaystyle \prod_{p\mid q}(1-p^{-s})$ where $q\in \N$ and $s\in \C$. Let $\alpha:=\alpha(x,y)$ be the saddle point of $y$-smooth numbers up to $x$. If $q$ is $y$-smooth with $P(q)\leq y$ and $\omega(q)\ll\sqrt{y}$, we have
    \[
    \Psi_q(x,y)\asymp g_q(\alpha)\Psi(x,y).
    \]
    In particular if $\omega(q)\ll \sqrt{y}/\log y$, we have
    \[
    \Psi_q(x,y)= g_q(\alpha)\Psi(x,y)(1+o(1)).
    \]
\end{lemma}
\subsection{Exponential Sums Estimates}
Besides the distribution of smooth numbers in arithmetic progressions, we also need an estimate for exponential sums over smooth numbers. This is needed for us for the transference principle. We record the results from \cite{H}. There are also results on exponential sums over smooth numbers over short intervals in \cite{MW}.

Let $R:=\log^{20}N$ and
\[
\mathfrak{M}:=\bigcup_{q\leq R}\bigcup_{(a,q)=1}[a/q-R/(qN),a/q+R/(qN)].
\]
This set is called the \textit{major arc}. The radius of each interval of this major arc set is $R/(qN)$. Clearly this is a subset of $[0,1]$. Its complement in $[0,1]$ is called the \textit{minor arc}. Later in Section \ref{sec:main}, we will work on $\Z/N\Z$ instead of $[N]$. We identify these two when working on exponential sums estimates. We also use the convention that if $a/N$ is in major or minor arc for $a\in\Z/N\Z$ means that $a/N$ is in major or minor arc when we view it as element of $[0,1]$. We first mention \cite[Proposition 5]{H} that gives an estimate of the exponential sum of smooth numbers in the minor arc.
\begin{lemma}
\label{lem:minor}
    For any large $\log^KN\leq y\leq N^{1/100}$ and any $\theta\in [0,1]\backslash\mathfrak{M}$, we have
    \[
    \sum_{\substack{n\leq N\\ n\in\mathcal{S}(y)}} e(n\theta)\ll \Psi(N,y)\frac{1}{\log^5N}.
    \]
\end{lemma}
To get an estimate for the major arc it turns out to be more difficult compare to the minor arc result. We follow closely the method in \cite[Section 4]{G}. The function we will use is function $f : \Z/N\Z\rightarrow \R$ with bounds
\begin{align}
\begin{split}
\label{eq:bound}
    \|f\|_{\infty}&=O(\log_2(N)/N^{\alpha}),\\
    \|f\|_1&=1+o(1).
\end{split}
\end{align}
We also have that if $L\geq N(\log N)^{-A-41}$ for some positive constant $A>0$ and if $X\subseteq [N]$ is an arithmetic progression $\{r,r+q,\dots,r+(L-1)q\}$ with $q\leq (\log N)^{20}$, we have
\[
\sum_{n\in X}f(n)=\frac{L}{N}\gamma_{r,q}(f)(1+o(1))
\]
where $\gamma_{r,q}(f)$ is a constant that depends only on $r$ and $q$ with $|\gamma_{r,q}|\leq q$. Here we identify $[N]$ with $\Z/N\Z$.

We recall several results from \cite[Section 4]{G} that will be useful to estimate the exponential sum over the major arc that we defined. For a residue class $r\mod q$, we denote $N_r$ for the set $\{n\leq N : n\equiv r\mod q\}$. For $\theta\in [0,1)$, we have
\[
\tau(\theta):=\frac{1}{N}\sum_{n\leq N} e(\theta n).
\]
\begin{lemma}
    Let $r$ be a residue class modulo $q$ and suppose $|\theta|\leq (\log N)^{20}/qN$. Suppose $f$ satisfies the estimates in \eqref{eq:bound}. Then we have
    \[
    \sum_{n\in N_r}f(n)e(n\theta)=q^{-1}\gamma_{r,q}(f)\tau(\theta)+o\left(q^{-1}\right).
    \]
\end{lemma}
\begin{lemma}
\label{lem:major}
    Suppose $f$ satisfies \eqref{eq:bound} and $\theta\in\mathfrak{M}_{a,q}$ for some $a,q$ with $\gcd(a,q)=1$ and $q\leq (\log N)^{20}$. Write
    \[
    \sigma_{a,q}(f)=\sum_{r\mod q} e(ar/q)\gamma_{r,q}(f).
    \]
    Then we have
    \[
    \widehat{f}(\theta)=q^{-1}\sigma_{a,q}(f)\tau(\theta-a/q)+o(1)
    \]
    with
    \[
    \widehat{f}(\theta)=\sum_{n\leq N}f(n)e(n\theta).
    \]
\end{lemma}

\subsection{Transference Principle}

The main ingredient to get our result is what we call \textit{transference principle}. We state the transference principle we use here, which is \cite[Proposition 5.2]{GT}. For other variations of the Fourier analytic transference principle, one can consult to \cite{Pren}.
\begin{prop}
\label{prop : transfer}
    Let $N$ be a large prime and let $0<\delta\leq 1$. Let $f : \Z/N\Z\rightarrow [0,\infty)$ and $\nu : \Z/N\Z\rightarrow [0,\infty)$ be any functions such that
    \begin{align*}
        f(n)\leq \nu(n),\quad n\in \Z/N\Z
    \end{align*}
    and
    \begin{align*}
        \frac{1}{N}\sum_{n\in \Z/N\Z} f(n)&\geq \delta.
    \end{align*}
    In addition, let $\eta\geq 0, M>0$, and $2<p<3$ be any parameters such that
    \begin{align*}
        \left| \frac{1}{N}\sum_{n\in \Z/NZ}\nu(n)e(an/N)-\mathbbm{1}_{a=0} \right|&\leq \eta,\quad a\in \Z/N\Z\\
        \sum_{a\in \Z/N\Z}\left| \frac{1}{N}\sum_{n\in \Z/N\Z} f(n)e(an/N) \right|^p&\leq M
    \end{align*}
    where $\mathbbm{1}_{a=0}$ is the indicator function when $a=0$.

    Then we have
    \[
    \frac{1}{N^2}\sum_{n,d\in \Z/N\Z}f(n)f(n+d)f(n+d)\geq c(\delta)-O_{\delta,p,M}(\eta)
    \]
    where $c(\delta)>0$ depends only on $\delta$. 
\end{prop}

We also record the following theorem, due to Harper \cite[Theorem 2]{H}.
\begin{theorem}
\label{thm:harp}
    There is an absolute constant $C>0$ such that the following is true.

    Let $p>2$ and suppose $x$ is large enough in terms of $p$, and let $\log^{C \max\{1,1/(p-2)\}} x\leq y\leq x$. Then for any complex numbers $(a_n)_{n\leq x}$ with absolute value at most $1$, we have
    \[
    \int_0^1\left| \sum_{\substack{n\leq x\\n\in \mathcal{S}(y)}} a_ne(n\theta) \right|^p\,d\theta\ll_p\frac{\Psi(x,y)^p}{x}
    \]
\end{theorem}
\section{Proof of Theorem \ref{thm:Roth Super Smooth}}
\label{sec:main}
We collect several results regarding the main ingredient for the proof of our main theorem, which is the $W$-trick. We give the proofs of those results in the next section. First, we define the functions we use for the $W$-trick.

Let $w=\frac12\log_3(N)$ and $W=\prod_{p\leq w}p$.
Let $b_1\in\mathcal{S}(w)$ that is at most $N^{1/2}$ and let $b_2\in [W]$ with $(b_2,W)=1$. If we denote $\mathfrak{b}=(b_1,b_2)$, we define the function $\nu_{\mathfrak{b}} : \Z\rightarrow [0,\infty)$ by
\[
\nu_{\mathfrak{b}}(n)=\begin{cases}
    2^{1/(1-\alpha)}\left(\prod_{p|W}\frac{1-p^{-1}}{1-p^{-\alpha}}\right)\frac{(Wn-b_2)^{1-\alpha}}{\alpha}&\quad\text{if }b_1(Wn-b_2)\in \mathcal{S}(N,y)\\
    0&\quad\text{otherwise}.
\end{cases}
\]
We denote $C_W:=\left(\prod_{p|W}\frac{1-p^{-1}}{1-p^{-\alpha}}\right)$ to make it short. Throughout we let $N_{\mathfrak{b}}'=2N_{\mathfrak{b}}+1$ and we extend $\nu_{\mathfrak{b}}$ to $\Z/N_{\mathfrak{b}}'\Z$ by defining $\nu_{\mathfrak{b}}$ to be $0$ outside $[N_{\mathfrak{b}}]$ viewed as subset of $\Z/N_{\mathfrak{b}}'\Z$. It is clear that $\nu_{\mathfrak{b}}$  is supported on $[N_{\mathfrak{b}}]\cap \mathcal{S}(N,y)$ where
\[
N_{\mathfrak{b}}:=\left\lfloor \frac{N}{b_1W} \right\rfloor+1.
\]
There exist such $b_1,b_2$ so that by combining Lemma \ref{lem:Granville} and \ref{lem : Breteche Tenenbaum result},
\[
\sum_{n\in\Z/N_{\mathfrak{b}}'\Z}\nu_{\mathfrak{b}}(n)\sim N_{\mathfrak{b}}'.
\]
To see this, we use summation by parts to get
\begin{align*}
    \sum_{n\in\Z/N_{\mathfrak{b}}'\Z}\nu_{\mathfrak{b}}(n)&\sim 2C_W\frac{(WN_{\mathfrak{b}}-b_2)^{1-\alpha}}{\alpha}\sum_{n\leq N_{\mathfrak{b}}} \mathbbm{1}_{\mathcal{S}(N,y)}(Wn-b_2)\\
    &\sim 2C_W(WN_{\mathfrak{b}})^{1-\alpha}\frac{N_{\mathfrak{b}}^{\alpha}}{C_WW^{1-\alpha}}\sim N_{\mathfrak{b}}'
\end{align*}
Now we record several results for the $W$-trick functions that we need in order to be able to apply Proposition \ref{prop : transfer}, whose proofs will be given at Section \ref{sec:trick}.
\begin{prop}
\label{prop : avg subset}
    There exists an absolute constant $C$ such that the following holds for any $\delta\geq C(\log N)^{-1}$. Let $A\subseteq \mathcal{S}(N,y)$ be such that $|A|=\delta \Psi(N,y)$. Then there are a $w$-smooth number $b_1\leq N^{1/2}$ and an integer $b_2\in [W]$ with $\gcd(b_2,W)=1$ such that
\[
\sum_n \mathbbm{1}_{A_{\mathfrak{b}}}(n)\nu_{\mathfrak{b}}(n)\gg \delta^{2-\alpha}N_{\mathfrak{b}}
\]
where
\[
A_{\mathfrak{b}}:=\{n\in \N : b_1(Wn-b_2)=x\text{ for some }x\in A\}.
\]
\end{prop}
\begin{prop}
\label{lem : avg}
    We have for each nonzero $a\in \Z/N_{\mathfrak{b}}'\Z$,
    \begin{align}
        \label{Fourier nu}
        \frac{1}{N_{\mathfrak{b}}'}\sum_{n\in \Z/N_{\mathfrak{b}}'\Z}\nu_{\mathfrak{b}}(n)e(an/N_{\mathfrak{b}'})=o(1).
    \end{align}
\end{prop}
\begin{prop}
\label{prop:p-th power}
    Let $f_{\mathfrak{b}}(n)=\mathbbm{1}_{A_{\mathfrak{b}}}(n)\nu_{\mathfrak{b}}(n)$. Then we have for $2<p<3$
    \[
    \sum_{a\in \Z/N_{\mathfrak{b}}'\Z}\left|\frac{1}{N_{\mathfrak{b}}'}\sum_{n\in \Z/N_{\mathfrak{b}}'\Z} f_{\mathfrak{b}}(n)e(an/N_{\mathfrak{b}}') \right|^p\ll W^{p(1-\alpha)}.
    \]
\end{prop}

We are ready to prove our main result by assuming the three results above. 
\begin{proof}[Proof of Theorem \ref{thm:Roth Super Smooth}]
Note that we have $\mathbbm{1}_{A_{\mathfrak{b}}}(n)=\frac{\alpha  f_{\mathfrak{b}(n)}}{C_W(Wn-b_2)^{1-\alpha}}$ where
\[
C_W=\prod_{p|W}\frac{1-p^{-1}}{1-p^{-\alpha}}.
\]
We have that
\[
C_W\asymp\frac{\exp(w^{1-\alpha}/\log w)}{\log w}
\]

Apply the estimates from Proposition \ref{prop : avg subset}, Proposition \ref{lem : avg}, and Propositions \ref{prop:p-th power} to the transference principle stated in Proposition \ref{prop : transfer} with $\eta=o(1), M=W^{p(1-\alpha)}$ and by taking $p=2.1$ in Proposition \ref{prop:p-th power}, we get
\begin{align*}
    &\sum_{n,d\in \Z/N_{\mathfrak{b}}'\Z}\mathbbm{1}_{A_{\mathfrak{b}}}(n)\mathbbm{1}_{A_{\mathfrak{b}}}(n+d)\mathbbm{1}_{A_{\mathfrak{b}}}(n+2d)\\&\gg \frac{1}{C_W^3(WN_{\mathfrak{b}}-b_2)^{3-3\alpha}}\sum_{n,d\in \Z/N_{\mathfrak{b}}'\Z}f_{\mathfrak{b}}(n)f_{\mathfrak{b}}(n+d)f_{\mathfrak{b}}(n+2d)\\
    &\gg \frac{N_{\mathfrak{b}}^2}{C_W^3(WN_{\mathfrak{b}}-b_2)^{3-3\alpha}}(c(\delta)-o_{\delta, N_{\mathfrak{b}}}(1)).
\end{align*}
Since the contribution from $d=0$ gives
\[
\sum_{n\in \Z/N_{\mathfrak{b}}'\Z}\mathbbm{1}_{A_{\mathfrak{b}}}(n)^3\ll N_{\mathfrak{b}}^{\alpha},
\]
we get
\begin{align*}
    &\sum_{n\in \Z/N_{\mathfrak{b}}'\Z}\sum_{d\in \Z/N_{\mathfrak{b}}'\Z,d\neq 0}\mathbbm{1}_{A_{\mathfrak{b}}}(n)\mathbbm{1}_{A_{\mathfrak{b}}}(n+d)\mathbbm{1}_{A_{\mathfrak{b}}}(n+2d)\\
    &\gg \frac{N_{\mathfrak{b}}^{3\alpha-1}}{C_W^3W^{3-3\alpha}}(c(\delta)-o_{\delta,N_{\mathfrak{b}}}(1))-N_{\mathfrak{b}}^{\alpha}\gg 1.
\end{align*}
Thus, there is a nontrivial arithmetic progression of length $3$ in $A_{\mathfrak{b}}$. Now if the difference of this progression in $A_{\mathfrak{b}}$ is $d$, then the arithmetic progression corresponds to another arithmetic progression of the same length in $A$ with difference $b_1Wd$, and this gives Theorem \ref{thm:Roth Super Smooth}.
\end{proof}
\section{$W$-trick}
\label{sec:trick}
In this section we prove the results regarding to the $W$-trick that are mentioned in Section \ref{sec:main}. The idea goes back to Green's paper \cite{G}, but the difficulty for smooth numbers is that they are not equidistributed in residue classes$\mod k$ for some positive integer $k\geq 2$. The functions mentioned in Section \ref{sec:main} and the proofs of the propositions follow closely that of the squarefull case as in \cite{BP}.

\begin{proof}[Proof of Proposition \ref{prop : avg subset}]
    We use Rankin's trick to get that the number of $y$-smooth numbers in $\{1,\dots,N\}$ divisible by a $w$-smooth number exceeding $N^{1/2}$ is at most
    \begin{align*}
        \sum_{\substack{b_1>N^{1/2}\\b_1\in\mathcal{S}(w)}}\frac{N^{\alpha}}{b_1}&\ll \sum_{b_1\in\mathcal{S}(w)}\frac{N^{\alpha}}{b_1}\left( \frac{b_1}{N^{1/2}} \right)^{\alpha}\\
        &\ll N^{\alpha/2}\sum_{b_1\in\mathcal{S}(w)}\frac{1}{b_1^{1-\alpha}}\\
        &\ll N^{\alpha/2+\varepsilon}
    \end{align*}
    for any $\varepsilon>0$.

    Therefore we have
    \[
\delta \Psi(N,y)=\sum_{\substack{b_1\leq N^{1/2}\\b_1\in\mathcal{S}(w)}}\sum_{\substack{b_2\in [W]\\(b_2,W)=1}} |A_{\mathfrak{b}}|+O\left( N^{\alpha/2+\varepsilon} \right).
\]
On the other hand, we also have
\[
\sum_{\substack{b_1\leq N^{1/2}\\b_1\in\mathcal{S}(w)}}\sum_{\substack{b_2\in [W]\\(b_2,W)=1}} |\mathcal{S}(N,y)_{\mathfrak{b}}|\ll \Psi(N,y).
\]
where
\[
\mathcal{S}(N,y)_{\mathfrak{b}}:=\{n\in \N : b_1(Wn-b_2)=x\text{ for some }x\in \mathcal{S}(N,y)\}.
\]
Therefore we have, by our assumption $\delta\geq C(\log N)^{-1},$
\[
\delta \sum_{\substack{b_1\leq N^{1/2}\\b_1\in\mathcal{S}(w)}}\sum_{\substack{b_2\in [W]\\(b_2,W)=1}} |\mathcal{S}(N,y)_{\mathfrak{b}}|\ll \sum_{\substack{b_1\leq N^{1/2}\\b_1\in\mathcal{S}(w)}}\sum_{\substack{b_2\in [W]\\(b_2,W)=1}} |A_{\mathfrak{b}}|.
\]
Applying the pigeonhole principle, there are $b_1,b_2$ such that
\[
\delta |\mathcal{S}(N,y)_{\mathfrak{b}}|\ll |A_{\mathfrak{b}}|.
\]

Now for such $b_1,b_2$, we throw away elements of $A_{\mathfrak{b}}$ that are smaller than $\delta N_{\mathfrak{b}}$ so we can assume $|A_{\mathfrak{b}}|>(\delta/2) N_{\mathfrak{b}}^{\alpha}$. By combining Lemma~\ref{lem:Granville} and Lemma~\ref{lem : Breteche Tenenbaum result} with $N_{\mathfrak{b}}$ replacing $N$, we get
\begin{align*}
    \sum_{n}\mathbbm{1}_{A_{\mathfrak{b}}}(n)\nu_{\mathfrak{b}}(n)&=\sum_n \mathbbm{1}_{A_{\mathfrak{b}}}(n)\frac{C_W(Wn-b_2)^{1-\alpha}}{\alpha}\\
    &=\sum_{\substack{m\leq N/b_1\\m=x/b_1,x\text{ is }y-\text{smooth}}} \frac{C_Wm^{1-\alpha}}{\alpha}\mathbbm{1}_{A_{\mathfrak{b}}}((m+b_2)/W)\\
    &\gg \frac{\delta}{2}\frac{(N/b_1)^{\alpha}}{W}(WN_{\mathfrak{b}}(\delta/2)-b_2)^{1-\alpha}\\
    &\gg \delta^{2-\alpha} \frac{N_{\mathfrak{b}}^{\alpha}}{W^{1-\alpha}} W^{1-\alpha}N_{\mathfrak{b}}^{1-\alpha}=\delta^{2-\alpha}N_{\mathfrak{b}}
    .
\end{align*}
\end{proof}

\begin{proof}[Proof of Proposition \ref{lem : avg}]
    We start with the easier case, which is the minor arc case. Using the estimate coming from Lemma \ref{lem:minor} and \cite[Lemma 7.5]{MV} , if $a'/N_{\mathfrak{b}}'$ is in the minor arc, we have after using summation by parts
    \begin{align*}
        &\sum_n\nu_{\mathfrak{b}}(n)e(a'n/N_{\mathfrak{b}}')\\
        &=\left(\prod_{p|W}\frac{1-p^{-1}}{1-p^{-\alpha}}\right)\sum_{\substack{m\leq N/b_1\\m\in\mathcal{S}(y)\\m+b_2\equiv 0\mod W}} \frac{m^{1-\alpha}}{\alpha}\mathbbm{1}_{A_{\mathfrak{b}}}\left( \frac{m+b_2}{W} \right) e\left( \frac{a'}{N_{\mathfrak{b}}'}\left( \frac{m+b_2}{W} \right) \right)\\
        &\ll \frac{\exp(w^{1-\alpha}/\log w)}{\log w}\frac{(N/b_1)^{1-\alpha}}{\alpha}\sum_{\substack{m\leq N/b_1\\m\in\mathcal{S}(y)\\m+b_2\equiv 0\mod W}} \mathbbm{1}_{A_{\mathfrak{b}}}\left( \frac{m+b_2}{W} \right) e\left( \frac{a'}{N_{\mathfrak{b}}'}\left( \frac{m+b_2}{W} \right) \right)\\
        &\ll \frac{\exp(w^{1-\alpha}/\log w)}{\log w} \frac{(N/b_1)^{1-\alpha}}{\alpha}\frac{1}{W\log^5N_{\mathfrak{b}}}= o(N_b).
    \end{align*}
    Thus, we get \eqref{Fourier nu} in the minor arc case.

    Now we move on to the case where $a'\neq 0$ and $a'/N_{\mathfrak{b}}'$ is in major arc. 

    To get the major arc case, it is more convenient to us to normalize the function $\nu_{\mathfrak{b}}$, that is, we define
    \[
    h(n)=\frac{1}{N_{\mathfrak{b}}'}\nu_{\mathfrak{b}}(n)
    \]
    for $n\in \Z/N_{\mathfrak{b}}'\Z$. In this case we have that $\|h\|_1=1+o(1)$ and $ \|h\|_{\infty}=O(W^{1+o(1)}/(N/b_1)^{\alpha})$. Using partial summation, if $X=\{r,r+q,\dots,r+(L-1)q\}\subseteq [N_{\mathfrak{b}}]$ considered as subset of $\Z/N_{\mathfrak{b}}'\Z$ and $L\geq N_{\mathfrak{b}}(\log N_{\mathfrak{b}})^{-A}$ for some positive real number $A$, we have that
    \[
    \sum_{n\in X}h(n)=\frac{L}{N_{\mathfrak{b}}}\left( \prod_{p|W}\frac{1-p^{-1}}{1-p^{-\alpha}}\prod_{p|Wq/d} \frac{1-p^{-\alpha}}{1-p^{-1}}+o(1) \right)
    \]
    where $d=\gcd(-b_2+Wr,Wq)$.

    Therefore, for this $h$ we have that
    \[
    \gamma_{r,q}(h)=\prod_{p|W}\frac{1-p^{-1}}{1-p^{-\alpha}}\prod_{p|Wq/d} \frac{1-p^{-\alpha}}{1-p^{-1}}.
    \]
    Now the next step is to calculate $\sigma_{a,q}(h)$.
    \begin{lemma}
        We have that
        \[
        \sigma_{a,q}(h)=\begin{cases}
            \displaystyle e\left( \frac{ab_2\overline{W}}{q} \right)\prod_{p|q}\frac{p^{-\alpha}-p^{-1}}{1-p^{-1}},&\quad\text{if }\gcd(q,W)=1\\
            0,&\quad\text{otherwise}.
        \end{cases}
        \]
    \end{lemma}
    \begin{proof}
        If we have that every prime factor of $q$ is at most $w$, we have $\gamma_{r,q}(h)=1$ for any residue class $r\mod q$. Therefore we have that $\sigma_{a,q}(h)=0$. Otherwise, $q$ has a prime factor greater than $w$. If $\gcd(q,W)>1$, we write $q=q_1q_2$ with $q_1$ being the product of prime powers that appear in $q$ with every prime is at most $w$, and $q_2$ is the product of prime powers of $q$ such that every prime is greater than $w$. We get that $\gcd(q_1,q_2)=1$. By multiplicativity, we have that
        \begin{align*}
            \sigma_{a,q}(h)&=\sum_{r\mod q}\left(e(ar/q)\prod_{\substack{p|(Wq/\gcd(-b_2+Wr,Wq))\\p>w}}\frac{1-p^{-\alpha}}{1-p^{-1}}\right)\\
            &=\sum_{r_2\mod q_2}e(ac_2r_2/q_2)\\
            &\times\sum_{r_1\mod q_1}\left(e(ac_1r_1/q_1)\prod_{\substack{p|(q/\gcd(-b_2+W(c_1r_1+c_2r_2),q))\\p>w}}\frac{1-p^{-\alpha}}{1-p^{-1}}\right)\\
            &=0
        \end{align*}
        for some nonzero integers $c_1$ and $c_2$ since by the Chinese Remainder Theorem, we have $\gcd(-b_2+Wr,Wq)$ depends only on $\gcd(-b_2+Wr,q)$.

        Now we assume that $\gcd(q,W)=1$. Thus, we have that every prime divisor of $q$ is greater than $w$.

        Dividing the sum according to the gcd, we have that
\[
\sigma_{a,q}(h)=\sum_{k|q}\sum_{\substack{r\\\gcd(-b_2+Wr,q)=k}} e(ar/q)\prod_{p|(q/k)}\frac{1-p^{-\alpha}}{1-p^{-1}}.
\]
Let $-b_2+Wr=t$, then in modulo $q$, we have that $r\equiv(t+b_2)\overline{W}(\mod q)$ where $\overline{W}$ is the inverse of $W$ modulo $q$. We can rewrite the sum as
\[
\sum_{k|q}\sum_{\substack{t\\\gcd(t,k)=k}} e\left(\frac{a\overline{W}(t+b_2)}{q}\right) \prod_{p|(q/k)}\frac{1-p^{-\alpha}}{1-p^{-1}}.
\]
We can change the order of the product and the summation to get
\[
\sum_{k|q}\prod_{p|(q/k)}\frac{1-p^{-\alpha}}{1-p^{-1}}e\left( \frac{ab_2\overline{W}}{q} \right)\sum_{\substack{t\\\gcd(t,q)=k}}e(at\overline{W}/q).
\]
The inner sum is Ramanujan's sum, using the fact $\gcd(a,q)=1$, it is $\mu(q/k)$, where $\mu$ is the M\"{o}bius function. Thus, we get
\[
\sigma_{a,q}(h)=\sum_{k|q}\prod_{p|(q/k)}\frac{1-p^{-\alpha}}{1-p^{-1}}e\left( \frac{ab_2\overline{W}}{q} \right)\mu(q/k).
\]
One can simplify this to
\[
e\left( \frac{ab_2\overline{W}}{q} \right)\prod_{p|q}\frac{p^{-\alpha}-p^{-1}}{1-p^{-1}}
\]
by using the multiplicative property of the inner function.
    \end{proof}
    Now using Lemma \ref{lem:major}, we have that if $a'/N$ is in the major arc, we have
    \[
    \widehat{h}(a'/N_{\mathfrak{b}}')=q^{-1}\left(\prod_{p|q}\frac{p^{-\alpha}-p^{-1}}{1-p^{-1}}\right)\tau(a'/N_{\mathfrak{b}}'-a/q)+o(1)
    \]
    if it is nonzero.
    
    Since $q>w$, we get that
    \[
    \widehat{h}(a'/N_{\mathfrak{b}}')=o(1)
    \]
    and we get what we want.
\end{proof}

\begin{proof}[Proof of Proposition \ref{prop:p-th power}]
    We have that
    \begin{align*}
        &\sum_{a\in \Z/N_{\mathfrak{b}}'\Z}\left|\frac{1}{N_{\mathfrak{b}}'}\sum_{n\in \Z/N_{\mathfrak{b}}'\Z} f_{\mathfrak{b}}(n)e(an/N_{\mathfrak{b}}') \right|^p\\
        &\ll \frac{C_W^p(WN_{\mathfrak{b}}-b_2)^{p(1-\alpha)}}{(\alpha N_{\mathfrak{b}})^p}\sum_{a\in \Z/N_{\mathfrak{b}}'\Z}\left| \sum_{n\in \Z/N_{\mathfrak{b}}'\Z} \mathbbm{1}_{A_{\mathfrak{b}}}(n)e(an/N_{\mathfrak{b}}') \right|^p\\
        &\ll \frac{C_W^p(WN_{\mathfrak{b}}-b_2)^{p(1-\alpha)}N_{\mathfrak{b}}}{(\alpha N_{\mathfrak{b}})^p}\int_0^1\left| \sum_{n\in A_{\mathfrak{b}}} e(\theta n) \right|^p\,d\theta.
    \end{align*}
    Now by using Theorem \ref{thm:harp}, we get that
    \begin{align*}
        \sum_{a\in \Z/N_{\mathfrak{b}}'\Z}\left|\frac{1}{N_{\mathfrak{b}}'}\sum_{n\in \Z/N_{\mathfrak{b}}'\Z} f_{\mathfrak{b}}(n)e(an/N_{\mathfrak{b}}') \right|^p&\ll \frac{C_W^p(WN_{\mathfrak{b}}-b_2)^{p(1-\alpha)}N_{\mathfrak{b}}}{(\alpha N_{\mathfrak{b}})^p} \frac{N_{\mathfrak{b}}^{\alpha p+o(1)}}{N_{\mathfrak{b}}}\\
        &\ll W^{p(1-\alpha)}.
    \end{align*}
\end{proof}
\section*{Acknowledgement}
The author would like to thank Fernando Xuancheng Shao for his advice and suggestions, also to Ali Alsetri and Muhammad Afifurrahman for helpful discussions. The author also thanks the anonymous referee for the suggestions and corrections.

\end{document}